\documentclass[12pt,a4paper, oneside,reqno]{article}
\usepackage[utf8]{inputenc}
\usepackage[english]{babel}
\usepackage[left=1in,right=1in,top=1in,bottom=1in]{geometry}
\usepackage{amsfonts,amssymb,amsmath,amsthm}

\usepackage[symbol]{footmisc} 
\usepackage{longtable}
\usepackage{tasks}
\usepackage{enumerate}
\usepackage{cite}
\usepackage{hyperref}
\usepackage{verbatim}
\usepackage{tkz-graph}
\usepackage{fancyhdr} 
\theoremstyle{plain}
\newtheorem{theorem}{Theorem}
\newtheorem{lemma}[theorem]{Lemma}
\newtheorem{proposition}[theorem]{Proposition}
\newtheorem{remark}[theorem]{Remark}

\newtheorem{definition}[theorem]{Definition}
\newtheorem{corollary}[theorem]{Corollary}
\usetikzlibrary{arrows}
\newtheorem*{maintheorem}{Main Theorem}

 \usepackage{fouriernc} 

\newcommand{\B}{\mathcal{B}}
\newcommand{\Z}{\mathbb{Z}}
\newcommand{\C}{\mathbb{C}}
\newcommand{\QDer}{\mathfrak{QDer}}
\newcommand{\GDer}{\mathfrak{GDer}}
\newcommand{\Der}{\mathfrak{Der}}
\newcommand{\Inn}{\mathrm{Inn}}
\newcommand{\Ann}{\mathrm{Ann}}
\newcommand{\AnnQ}{\mathrm{Ann}^{Q}}
\newcommand{\End}{\mathrm{End}}
\newcommand{\ad}{\mathrm{ad}}
\newcommand{\id}{\mathrm{id}}
\newcommand{\dhat}{\hat{\delta}}
\newcommand{\ac}[2]{a^{#1}_{#2}}
\newcommand{\acp}[2]{a^{\prime\,#1}_{#2}}

\begin{document}

\noindent{\Large 
Quasi-derivations of Block Lie algebras $\mathcal{B}(q)$}
 \footnote{The authors thank the SRMC (Sino-Russian Mathematics Center at Peking University, Beijing, China) for its hospitality and excellent working conditions, where some parts of this work were done. 
The work is supported by FCT 2023.08031.CEECIND and UID/00212/2025.}  
\footnote{The original concept, theoretical framework, and primary results were developed by I.K., A.K., and Z.S.; 
while Aristotle (https://aristotle.harmonic.fun/) acted as a non-human research assistant, assisting in error verification and proof simplification.}
  \bigskip
 
	 \bigskip

\begin{center}	
	{\bf
		Ivan Kaygorodov\footnote{CMA-UBI, University of  Beira Interior,  Portugal;  \    kaygorodov.ivan@gmail.com},
   Abror Khudoyberdiyev\footnote{V.I.Romanovskiy Institute of Mathematics, Uzbekistan Academy of Sciences; National University of Uzbekistan, Uzbekistan; \ 
khabror@mail.ru} \&
Zarina Shermatova\footnote{V.I.Romanovskiy Institute of Mathematics, Uzbekistan Academy of Sciences; Kimyo International University in Tashkent, Uzbekistan;  \ 
zarina\_shermatova91@mail.ru}}   
\end{center}

 \bigskip

   \bigskip

\noindent {\bf Abstract.}
{\it 
We compute quasi-derivations of Block Lie algebras $\B(q)$. In particular, we prove that $\B(q)$ has various nontrivial quasi-derivations, i.e., quasi-derivations different from derivations and $\delta$-derivations. As corollaries, we obtain a complete description of $\delta$-derivations, $\delta$-biderivations, and transposed $\delta$-Poisson structures on $\B(q)$ for $\delta\notin\{0;1\}$. In particular, for $\delta\notin\{0;1\}$, $\B(q)$ up to isomorphism admits only one nontrivial transposed $\delta$-Poisson structure if $q\in\Z$, and $\B(q)$ does not admit nontrivial transposed $\delta$-Poisson structures if $q\notin\Z$.
}
 \bigskip


 \bigskip
 
 \bigskip


  \bigskip \bigskip

\tableofcontents 
 \bigskip
  \bigskip \bigskip
\newpage

\section*{Introduction}
The notion of reverse derivations, as a particular case of Jordan derivations, was introduced in a paper by Herstein in 1957 \cite{12}. In the anticommutative case, the notion of reverse derivations gives antiderivations, which were actively studied by Filippov \cite{9}. For example, he proved that any prime Lie algebra with a nonzero antiderivation satisfies the standard identity of degree 5. Later, Filippov generalized the notion of derivations and antiderivations to $\delta$-derivations \cite{10}, which have been actively studied for algebras of various varieties (see \cite{14,30} and references therein). Each nontrivial $\delta$-derivation of a Lie algebra gives a Hom-Lie structure \cite{10}. $\frac{1}{2}$-derivations play an important role in the definition of transposed Poisson algebras \cite{2,8,13,16,17,22,24}. Similarly, the notion of $\delta$-derivations plays an important role in the definition of transposed $\delta$-Poisson algebras \cite{1,6}. A notion generalizing $\delta$-derivations was introduced by Leger and Luks under the name of generalized derivations in 2000 \cite{19}. Generalized derivations of triangular, alternative, Jordan, and $n$-ary algebras were studied by Jim\'enez-Gestal and P\'erez-Izquierdo; Shestakov; Mart\'{\i}n Barquero, Mart\'{\i}n Gonz\'alez, S\'anchez-Ortega, and Vandeyar; and Kaygorodov and Popov (see \cite{14} and references therein). On the other hand, generalized derivations of Lie algebras provide a particular case of near-derivations defined by Bre\v{s}ar, and generalized derivations are closely related to Jordan $\{g,h\}$-derivations introduced by Bre\v{s}ar. Generalized derivations have appeared in various contexts. For instance, they were used by Andruszkiewicz, Brzezi\'nski, and Radziszewski to construct Lie algebras and by Jim\'enez-Gestal and P\'erez-Izquierdo in the study of real division algebras. Burde and Dekimpe used generalized derivations of Lie algebras to study post-Lie algebra structures. Recently, Kaygorodov, Khudoyberdiyev, and Shermatova described all generalized derivations of the Witt algebra and of the algebras $W(a,b)$ \cite{18}. As was proved there, the algebras $W(a,b)$ give nontrivial examples of transposed $\delta$-Poisson structures for an arbitrary $\delta\in\C$. These ideas were extended to the deformed Heisenberg--Virasoro algebra \cite{21} and the $n$-th Schr\"odinger superalgebra \cite{26}. The results obtained in \cite{18,21,26} motivated the search for new Lie algebras that admit nontrivial transposed $\delta$-Poisson structures.

The Witt algebra and its generalizations are among the most popular infinite-dimensional Lie algebras. The Witt algebra plays a foundational role in the construction of other important algebraic structures (see, for example, \cite{3,8,18}). One of the most interesting generalizations of the Witt algebra was introduced by Block in 1958 \cite{4}. Since then, Block-type algebras have been of considerable research interest. Osborn, Zhao, Xu, and Xia constructed new examples of Block-type algebras in \cite{23,27,29}. \DJ okovi\'c and Zhao described derivations and the second cohomology of generalized Block algebras \cite{7}. Modules and representations over various classes of Block-type algebras, Block-type superalgebras, and Block-type conformal algebras and superalgebras were studied in papers by Chen, Guo, Zhao, Wang, Xia, Yuan, Liu, and Su (see, for example, \cite{5,11,25} and references therein). Various linear and bilinear maps of Block-type algebras and superalgebras have also been studied: Liu, Guo, and Zhao described biderivations of Block Lie algebras in 2018 \cite{20}; Kaygorodov and Khrypchenko described transposed Poisson structures on Block Lie algebras and superalgebras in 2023 \cite{15}; Huang, Gao, and Sun described transposed Poisson structures on the twisted $N=1$ Block-Lie superalgebra in 2025 \cite{13}. The present paper is a continuation of the cited works and is devoted to the complete description of quasi-derivations of the Block algebras $\B(q)$, which generalizes some previously obtained results in \cite{15}.

\section{Preliminaries}
In this section, we recall some definitions and known results for studying generalized derivations and quasi-derivations of Lie algebras. All algebras and vector spaces are considered over the complex field. Let $S$ be a set of vectors; we denote the linear span of $S$ by $\langle S\rangle$. By $\delta_{i,j}$ we denote the standard Kronecker symbol and $\dhat_{i,j}:=1-\delta_{i,j}$.

\subsection{Main definitions}
\begin{definition}
Let $\mathfrak{L}$ be a Lie algebra. The linear map $D:\mathfrak{L}\to\mathfrak{L}$ is called a derivation if
\[ D[x,y]=[D(x),y]+[x,D(y)]. \]
\end{definition}
As usual, we denote the set of all derivations of $\mathfrak{L}$ by $\Der(\mathfrak{L})$. Obviously, $\Der(\mathfrak{L})$ is a Lie subalgebra of the general linear algebra $\mathfrak{gl}(\mathfrak{L})$. There are several generalizations of the notion of a derivation. In this paper, we consider generalized derivations and quasi-derivations of Lie algebras as defined by Leger and Luks in \cite{19}.

\begin{definition}
Let $\mathfrak{L}$ be a Lie algebra and $f:\mathfrak{L}\to\mathfrak{L}$ be a linear map. If there exist linear maps $f',f'':\mathfrak{L}\to\mathfrak{L}$ such that
\[ [f(x),y]+[x,f'(y)]=f''[x,y], \]
then $f$ is called a generalized derivation. If $f=f'$, then $f$ is called a quasi-derivation. If $f=f'=\delta f''$, we say that $f''$ is a $\delta$-derivation.
\end{definition}
By $\GDer(\mathfrak{L})$ we will denote the set of all generalized derivations of $\mathfrak{L}$; by $\QDer(\mathfrak{L})$ the set of all quasi-derivations of $\mathfrak{L}$; by $\Der_{\delta}(\mathfrak{L})$ the set of all $\delta$-derivations (for a fixed $\delta$) of $\mathfrak{L}$; by $\Der_{[\delta]}(\mathfrak{L})$ the linear span of all $\delta$-derivations with $\delta\neq0$ of $\mathfrak{L}$; by $\Inn(\mathfrak{L})$ the set of all inner derivations of $\mathfrak{L}$. Obviously, $\Inn(\mathfrak{L})$, $\Der(\mathfrak{L})$, $\QDer(\mathfrak{L})$, and $\GDer(\mathfrak{L})$ are Lie subalgebras of $\mathfrak{gl}(\mathfrak{L})$, $\Der_{[\delta]}(\mathfrak{L})$ is a subspace of $\mathfrak{gl}(\mathfrak{L})$, and
\[ \Inn(\mathfrak{L})\subseteq\Der(\mathfrak{L})\subseteq\Der_{[\delta]}(\mathfrak{L})\subseteq\QDer(\mathfrak{L})\subseteq\GDer(\mathfrak{L})\subseteq\mathfrak{gl}(\mathfrak{L}). \]
(If $\varphi$ is a $\delta$-derivation with $\delta\neq0$, then $\delta\varphi$ is a quasi-derivation with related map $\varphi$. A $0$-derivation is just a linear map vanishing on $[\mathfrak{L},\mathfrak{L}]$; in general it is not a quasi-derivation, which is why $\delta=0$ is excluded above. Also, the set-theoretic union of all $\Der_\delta(\mathfrak{L})$ is not closed under addition in general.)

\begin{definition}[see \cite{1,6}]
Let $\delta$ be a fixed complex number. An algebra $(\mathfrak{P},\cdot,[\cdot,\cdot])$ is defined to be a transposed $\delta$-Poisson algebra if $(\mathfrak{P},\cdot)$ is a commutative associative algebra, $(\mathfrak{P},[\cdot,\cdot])$ is a Lie algebra, and the following identity holds:
\begin{equation}\label{tdp}
x\cdot[y,z]=\delta\big([x\cdot y,z]+[y,x\cdot z]\big).
\end{equation}
A Lie algebra $(\mathfrak{L},[\cdot,\cdot])$ has a nontrivial transposed $\delta$-Poisson structure $\cdot$ if $(\mathfrak{L},\cdot,[\cdot,\cdot])$ is a transposed $\delta$-Poisson algebra with both multiplications nonzero.
\end{definition}
The description of $\delta$-derivations (or, more generally, $\delta$-biderivations) gives a constructive way to obtain a classification of transposed $\delta$-Poisson structures on a given Lie algebra. Namely, the following obvious statement is a generalization of a result from \cite{8}.

\begin{lemma}\label{lem4}
Let $(\mathfrak{L},\cdot,[\cdot,\cdot])$ be a transposed $\delta$-Poisson algebra and $z$ an arbitrary element of $\mathfrak{L}$. Then the right multiplication $R_z$ in the associative commutative algebra $(\mathfrak{L},\cdot)$ is a $\delta$-derivation of the Lie algebra $(\mathfrak{L},[\cdot,\cdot])$. In particular, if $\Der_{\delta}(\mathfrak{L})\subseteq\langle\id\rangle$, then there are no nontrivial transposed $\delta$-Poisson structures on $\mathfrak{L}$.
\end{lemma}

\subsection{Block Lie algebras $\B(q)$}
Xia, You, and Zhou in \cite{28} defined the Block algebra $\B(q)$ for a fixed complex number $q\in\C$ as the algebra with basis $\{L_{m,i}\mid m,i\in\Z\}$ and the following multiplication table:
\begin{equation}\label{mult}
[L_{m,i},L_{n,j}]=\big(n(i+q)-m(j+q)\big)L_{m+n,i+j},\qquad i,j,m,n\in\Z.
\end{equation}
In their study, it was established that the algebras $\B(q_1)$ and $\B(q_2)$ are non-isomorphic for $q_1\neq\pm q_2$; note that $\B(q)\cong\B(-q)$ for every $q\in\C$, an isomorphism being given by $L_{m,i}\mapsto L_{-m,-i}$. Furthermore, they characterized the automorphism group and the derivation algebra of $\B(q)$ for any $q\in\C$. Their results also showed that the second scalar cohomology group of $\B(q)$ is one-dimensional, leading to the identification of the unique nontrivial central extension of the algebra.

\begin{proposition}[see \cite{15,20}]
For the algebra $\B(q)$, the following statements hold:
\begin{align*}
\AnnQ\big(\B(q)\big) &= \big\{g\in\End\big(\B(q)\big)\mid g\big(\B(q)\big)\subseteq\Ann\big(\B(q)\big)\big\}\\
&= \big\langle g_{d,k}\mid g_{d,k}(L_{m,i})=\delta_{d,m}\delta_{k,i}L_{0,-q}\big\rangle\subseteq\QDer\big(\B(q)\big),\\
\Der\big(\B(q)\big) &= \begin{cases}
\Inn\big(\B(q)\big)\oplus\langle D_1,D_2,D_5\rangle, & \text{if } q=0,\\
\Inn\big(\B(q)\big)\oplus\langle D_1,D_3,D_4,D_5\rangle, & \text{if } q\in\Z^{*},\\
\Inn\big(\B(q)\big)\oplus\langle D_1\rangle, & \text{if } q\notin\Z,
\end{cases}\\
\Der_{\frac12}\big(\B(q)\big) &= \begin{cases}
\langle\id\rangle, & \text{if } q\notin\Z,\\
\langle\id,D_5\rangle, & \text{if } q\in\Z,
\end{cases}
\end{align*}
where the maps $D_i$ are defined by
\begin{align*}
D_1(L_{m,i})&=iL_{m,i}, & D_2(L_{m,i})&=mL_{m,i}, & D_3(L_{m,i})&=mL_{m,i-q},\\
D_4(L_{m,i})&=(q+i)L_{m,i-q}, & D_5(L_{0,-2q})&=L_{0,-q},
\end{align*}
and, for $q\in\Z$, $D_5(L_{m,i})=0$ for $(m,i)\neq(0,-2q)$ (the maps $D_3$, $D_4$, $D_5$, and $g_{d,k}$ are considered only for $q\in\Z$). Here $\Ann\big(\B(q)\big)$ denotes the annihilator (the center) of $\B(q)$; it equals $\langle L_{0,-q}\rangle$ if $q\in\Z$ and is zero if $q\notin\Z$, so that $\AnnQ\big(\B(q)\big)=0$ for $q\notin\Z$. Note that an element of $\AnnQ\big(\B(q)\big)$ is an arbitrary, possibly infinite, formal sum $\sum_{d,k}\sigma_{d,k}g_{d,k}$.
In particular,
\begin{align*}
&\ad_{L_{0,-q}}=0 \text{ if } q\in\Z;\\
&D_2\in\Inn\big(\B(q)\big) \text{ if } q\neq0; \qquad D_5\in\Der_{\delta}\big(\B(q)\big) \text{ for every } \delta\in\C\text{, if } q\in\Z.
\end{align*}
\end{proposition}

\begin{proposition}
The linear maps $p_k$ and $h_{d,k}$, defined as follows,
\begin{align*}
p_k(L_{m,i}) &= (i+q)L_{m,k+i-1},\\
h_{d,k}(L_{m,i}) &= mL_{d+m-1,k+i},
\end{align*}
are quasi-derivations of $\B(q)$ with related maps
\begin{align*}
p'_k(L_{m,i}) &= (k+i+2q-1)L_{m,k+i-1},\\
h'_{d,k}(L_{m,i}) &= (d+m-1)L_{d+m-1,k+i}.
\end{align*}
In particular, $p_1=D_1+q\,\id$; \ $p_{1-q}=D_4$ for $q\in\Z^{*}$;
\ $h_{1,0}=D_2$, \ $h_{1,-q}=D_3$ for $q\in\Z^{*}$, and $h_{1,k}=\frac{1}{k+q}\ad_{L_{0,k}}$ for $k\neq-q$.
\end{proposition}
\begin{proof}
The proof follows from direct computations.
\end{proof}

\section{Quasi-derivations of Block Lie algebras}
Let $f\in\QDer\big(\B(q)\big)$. Then there exists a linear map $f'$ such that the following holds for all $x,y\in\B(q)$:
\[ [f(x),y]+[x,f(y)]=f'([x,y]). \]
Let the action of $f$ and $f'$ on the basis elements $L_{m,i}$ be given by
\[ f(L_{m,i})=\sum_{d,k\in\Z}\ac{d,k}{m,i}L_{d,k} \quad\text{and}\quad f'(L_{m,i})=\sum_{d,k\in\Z}\acp{d,k}{m,i}L_{d,k}. \]
By substituting these into the defining identity of a quasi-derivation, using \eqref{mult}, and comparing the coefficients of the basis elements, we obtain the following general constraint: 
\medskip 

{$\big(n(k-j+q)-(d-n)(j+q)\big)\ac{d-n,k-j}{m,i}+\big((d-m)(i+q)-m(k-i+q)\big)\ac{d-m,k-i}{n,j}=$
\begin{equation}\label{main}
=\big(n(i+q)-m(j+q)\big)\acp{d,k}{m+n,i+j}.
\end{equation}

\begin{remark}
Relation \eqref{main} only involves coefficients $\ac{d,k}{m,i}$ and $\acp{d,k}{m,i}$ with the same difference $(d-m,k-i)$, i.e., the defining identity of a quasi-derivation is homogeneous with respect to the natural $\Z^2$-grading of $\B(q)$. Hence every homogeneous component $f_{(a,b)}(L_{m,i})=\ac{m+a,i+b}{m,i}L_{m+a,i+b}$ of a quasi-derivation $f$ is again a quasi-derivation, and it suffices to determine the homogeneous quasi-derivations. In the proofs below, some exceptional values of the indices (such as $m=-1$ or $i=-1$) have to be treated separately; we indicate how to do this where it is needed.
\end{remark}

\subsection{Quasi-derivations of $\B(0)$}\label{th7-sub}
\begin{theorem}\label{th7}
$\QDer\big(\B(0)\big)=\langle\ad_{L_{d,k}}\mid d\neq0,\ k\in\Z\rangle\oplus\langle\id\rangle\oplus\langle p_k\rangle\oplus\langle h_{d,k}\rangle\oplus\AnnQ\big(\B(0)\big)$.
\end{theorem}
\begin{proof}
Let $f\in\QDer\big(\B(0)\big)$. Then the relations \eqref{main} have the following form:
\begin{equation}\label{q0}
(nk-dj)\ac{d-n,k-j}{m,i}+(di-mk)\ac{d-m,k-i}{n,j}=(ni-mj)\acp{d,k}{m+n,i+j}.
\end{equation}
Since linear combinations of quasi-derivations are quasi-derivations, we may replace $f$ with the normalized map
\[ \tilde f=f-\sum_{d\neq0}\beta_{d,k}\ad_{L_{d,k}}-\gamma\,\id-\sum_k\lambda_kp_k-\sum_d\mu_{d,k}h_{d,k}-\sum_{d,k}\sigma_{d,k}g_{d,k}. \]
By choosing the coefficients as follows:
\begin{align*}
\beta_{d\neq0,k} &= -\ac{d,k+1}{0,1}d^{-1};\\
\gamma &= \tfrac12\big(\ac{1,0}{1,0}+\ac{-1,0}{-1,0}\big);\\
\lambda_k &= \delta_{k,0}\big(\ac{1,0}{1,1}-\ac{1,-1}{1,0}\big)+\delta_{k,1}\big(\ac{0,1}{0,1}-\tfrac12(\ac{1,0}{1,0}+\ac{-1,0}{-1,0})\big)+\dhat_{k,0}\dhat_{k,1}\ac{0,k}{0,1};\\
\mu_{d,k} &= \delta_{d,0}\delta_{k,0}\big(\ac{0,1}{1,1}-\ac{-1,1}{0,1}\big)+\tfrac12\delta_{d,1}\delta_{k,0}\big(\ac{1,0}{1,0}-\ac{-1,0}{-1,0}\big)\\
&\quad+\delta_{d,1}\dhat_{k,0}\ac{1,k}{1,0}+\dhat_{d,1}\dhat_{(d,k),(0,0)}\Big(\ac{d,k}{1,0}+\frac{k}{d-1}\ac{d-1,k+1}{0,1}\Big);\\
\sigma_{d,k} &= \delta_{d,0}\delta_{k,0}\big(\ac{0,0}{0,0}-\tfrac12(\ac{1,0}{1,0}+\ac{-1,0}{-1,0})\big)+\delta_{d,0}\dhat_{k,0}\big(\ac{0,0}{0,k}-k\lambda_{1-k}\big)\\
&\quad+\dhat_{d,0}\big(\ac{0,0}{d,k}-d\mu_{1-d,-k}\big),
\end{align*}
we may assume without loss of generality that
\begin{equation}\label{n0}
\ac{-1,0}{-1,0}=\ac{1,0}{1,1}=\ac{0,1}{1,1}=\ac{0,0}{d,k}=\ac{d,k}{1,0}=\ac{d,k}{0,1}=0\quad\text{for all}\quad d,k\in\Z.
\end{equation}
We now show that \eqref{n0} implies $\ac{d,k}{m,i}=0$ for all $m,i,d,k\in\Z$.

\begin{description}
\item[Step 1.] First, we show that $\ac{d,k}{m,i}=0$ for $d\neq m$ and $k\neq i$.

Taking $m=n=i=0$ and $j=1$ in \eqref{q0}, we obtain $\ac{d\neq0,k}{0,0}=0$.

Taking $m=i=j=0$ and $n=1$ in \eqref{q0}, we get $\ac{d,k\neq0}{0,0}=0$, i.e., $\ac{d,k}{0,0}=0$ for all $d,k\in\Z$.

Taking $i=j=0$ and $n=1$ in \eqref{q0}, we obtain $k\ac{d-1,k}{m,0}-mk\ac{d-m,k}{1,0}=0$, i.e., $\ac{d,k\neq0}{m,0}=0$.

Taking $m=n=0$ and $j=1$ in \eqref{q0}, we obtain $d\ac{d,k-1}{0,i}-di\ac{d,k-i}{0,1}=0$, i.e., $\ac{d\neq0,k}{0,i}=0$.

Considering $n=0$ and $i=0$ in \eqref{q0}, for any $m$ and $j$, we get
\[ dj\ac{d,k-j}{m,0}+mk\ac{d-m,k}{0,j}=mj\acp{d,k}{m,j}, \]
which implies $\acp{d,k}{m,j}=0$ for $m\notin\{0;d\}$ and $j\notin\{0;k\}$.

Setting $d=m$, $k=i$, $n\neq0$, and $j\neq0$ in \eqref{q0}, we obtain
\[ (ni-mj)\ac{m-n,i-j}{m,i}=(ni-mj)\acp{m,i}{m+n,i+j}. \]
By this equality, we have $\ac{m-n,i-j}{m,i}=0$ for $n\notin\{0;-m\}$, $j\notin\{0;-i\}$, and $ni\neq mj$.

Finally, setting $d=2m+n$ and $k=2i+j$ in \eqref{q0}, we obtain that $\ac{2m,2i}{m,i}=0$ for $m,i\neq0$.

The previous two relations do not cover the cases $d=2m$, $k=2i$, or $di=mk$. All cases can be treated uniformly as follows. Taking $n=0$ in \eqref{q0} and using $\ac{d\neq0,k}{0,j}=0$, we get $dj\,\ac{d,k-j}{m,i}=mj\,\acp{d,k}{m,i+j}$ for $d\neq m$; choosing $j\notin\{0;-i\}$ and using $\acp{d,k}{m,i+j}=0$ for $m\notin\{0;d\}$, $i+j\notin\{0;k\}$, we obtain $\ac{d,k}{m,i}=0$ for $d\neq0$, $m\notin\{0;d\}$, $k\neq i$. For $d=0$, taking $j=0$ in \eqref{q0} and using $\ac{d-m,k-i}{n,0}=0$ for $k\neq i$, we get $nk\,\ac{d-n,k}{m,i}=ni\,\acp{d,k}{m+n,i}$; with $d=n\notin\{0;-m\}$ this gives $\ac{0,k}{m,i}=0$ for $m\neq0$, $i\neq0$, $k\notin\{0;i\}$. The remaining coefficients $\ac{0,0}{m,i}$, $\ac{d,k}{m,0}$ ($k\neq0$), and $\ac{d,k}{0,i}$ ($d\neq0$) vanish by \eqref{n0} and the relations above.
Thus, we conclude that $\ac{d,k}{m,i}=0$ for $d\neq m$, $k\neq i$.

\item[Step 2.] Second, we show that $\ac{d,i}{m,i}=0$ for $d\neq m$.
We consider the relation \eqref{q0} with the restriction $k=i+j$. Taking $n=i=0$ and $k=j$ in \eqref{q0}, we obtain
\[ dj\ac{d,0}{m,0}+mj\ac{d-m,j}{0,j}=mj\acp{d,j}{m,j}. \]
Since for $d\neq m$, we have $\ac{d-m,j}{0,j}=0$, it follows that $d\ac{d,0}{m,0}=m\acp{d,j}{m,j}$ for $d\neq m$ and $j\neq0$.
Then, setting $n=0$, $j=1$, and $k=i+1$ in \eqref{q0}, we get
\[ d\ac{d,i}{m,i}-\big(di-m(i+1)\big)\ac{d-m,1}{0,1}=m\acp{d,i+1}{m,i+1}. \]
Since $\ac{d-m,1}{0,1}=0$ by \eqref{n0}, we have $d\ac{d,i}{m,i}=m\acp{d,i+1}{m,i+1}$.
Similarly, setting $n=0$, $j=-1$, and $k=i-1$ in \eqref{q0}, we get $d\ac{d,i}{m,i}=m\acp{d,i-1}{m,i-1}$ for $d\neq m$.
From the previous three equalities, we derive
\begin{equation}\label{six}
\ac{d,i}{m,i}=\ac{d,0}{m,0}\quad\text{for}\quad d\notin\{0;m\}.
\end{equation}
(For $i\neq-1$ one uses the second equality, and for $i=-1$ the third one.)

In particular, in the case $m=1$, we get $\ac{d,i}{1,i}=\ac{d,0}{1,0}=0$ for $d\notin\{0;1\}$.
After that, by setting $n=j=1$, $i=m\neq1$, and $k=m+1$ in \eqref{q0}, for any $d\neq m+1$, we obtain
\[ (m+1-d)\big(\ac{d-1,m}{m,m}-m\ac{d-m,1}{1,1}\big)=0, \]
which implies $\ac{d-1,m}{m,m}=m\ac{d-m,1}{1,1}$ for $m\notin\{1;d-1\}$.
Since $\ac{d,1}{1,1}=0$ for $d\neq1$, we derive that $\ac{d,m}{m,m}=0$ for $d\neq m$.
Thus, using \eqref{six}, we derive $\ac{d,0}{m,0}=\ac{d,m}{m,m}=0$, which implies $\ac{d,i}{m,i}=0$ for $d\notin\{0;m\}$.

Putting $n=d=1$, $i=0$, and $k=j$ in \eqref{q0}, we obtain
\[ mj\ac{1-m,j}{1,j}=mj\acp{1,j}{m+1,j},\ \text{which implies } \acp{1,j}{m+1,j}=0 \text{ for } j\neq0,\ m\notin\{0;1\}. \]
Setting $n=d=1$, $j=0$, and $k=i$ in \eqref{q0}, we obtain
\[ i\ac{0,i}{m,i}=i\acp{1,i}{m+1,i},\ \text{which yields } \ac{0,i}{m,i}=0 \text{ for } m\notin\{0;1\}. \]
Finally, taking $n=2$, $m=1$, $d=2$, $j=2i$, and $k=3i$ in \eqref{q0}, we obtain
\[ 2\ac{0,i}{1,i}=\ac{1,2i}{2,2i},\ \text{which implies } \ac{0,i}{1,i}=0. \]

\item[Step 3.] Third, we show that $\ac{m,k}{m,i}=0$ for $k\neq i$.
First, setting $n=0$, $i=0$, and $d=m$ in \eqref{q0}, we get
\[ mj\ac{m,k-j}{m,0}+mk\ac{0,k}{0,j}=mj\acp{m,k}{m,j}. \]
Consequently, since $\ac{m,k-j}{m,0}=0$ for $k\neq j$, we obtain $k\ac{0,k}{0,j}=j\acp{m,k}{m,j}$ for $m\neq0$, $k\neq j$.
Then, considering $n=1$, $j=0$, and $d=m+1$ in \eqref{q0}, by \eqref{n0}, we have
\[ k\ac{m,k}{m,i}=i\acp{m+1,k}{m+1,i}. \]
By combining the last two relations, we obtain the following identity:
\[ \ac{m,k}{m,i}=\ac{0,k}{0,i},\quad m\neq-1,\ k\notin\{0;i\};\ \text{i.e., if } i=1, \text{ we get } \ac{m,k}{m,1}=0 \text{ for } k\notin\{0;1\},\ m\neq-1. \]
Next, by setting $n=j=1$, $i=m$, and $d=m+1$ in \eqref{q0}, for any $k\neq m+1$, we obtain
\[ (m+1-k)\big(\ac{m,k-1}{m,m}-m\ac{1,k-m}{1,1}\big)=0,\ \text{which implies } \ac{m,k-1}{m,m}=m\ac{1,k-m}{1,1}. \]
Since $\ac{1,k}{1,1}=0$ for $k\neq1$, we derive that $\ac{m,k}{m,m}=0$ for $k\neq m$.
Thus, using $\ac{0,k}{0,i}=\ac{i,k}{i,i}=0$ (for $i\neq-1$), we conclude that $\ac{m,k}{m,i}=0$ for $k\notin\{0;i\}$, $m\neq-1$, $i\neq-1$.

The cases $m=-1$ or $i=-1$ are treated as follows. Relation \eqref{q0} with $n=0$, $d=m$, $j=1$ gives, by \eqref{n0}, $\ac{m,k-1}{m,i}=\acp{m,k}{m,i+1}$ for $m\neq0$. Using this, relation \eqref{q0} with $n=0$, $d=m$ (respectively, with $m=0$, $d=n$; respectively, with $m=-1$, $n=2$, $d=1$) yields
\begin{align*}
& j\ac{m,k-j}{m,i}+(k-i)\ac{0,k-i}{0,j}=j\ac{m,k-1}{m,i+j-1}, && m\neq0,\\
& (k-j)\ac{0,k-j}{0,i}+i\ac{n,k-i}{n,j}=i\ac{n,k-1}{n,i+j-1}, && n\neq0,\\
& (2k-j)\ac{-1,k-j}{-1,i}+(i+k)\ac{2,k-i}{2,j}=(2i+j)\ac{1,k-1}{1,i+j-1}.
\end{align*}
Taking $i=-1$ and $|j|$ large enough in the first relation, we get $\ac{m,K}{m,-1}=0$ for $m\notin\{0;-1\}$, $K\neq-1$; taking $i=-1$, $n=1$, and $|j|$ large enough in the second one, we get $\ac{0,K}{0,-1}=0$ for $K\notin\{0;-1\}$; finally, taking $|j|$ large enough in the third relation, we get $\ac{-1,K}{-1,i}=0$ for all $K\neq i$. In all three cases, the other coefficients in the relation vanish by the cases already proved.

After that, putting $n=0$, $i=k=1$, and $d=m$ in \eqref{q0}, we obtain
\[ mj\ac{m,1-j}{m,1}=mj\acp{m,1}{m,j+1},\ \text{which implies } \acp{m,1}{m,j+1}=0 \text{ for } m\neq0,\ j\notin\{0;1\}. \]
Setting $n=0$, $j=k=1$, and $d=m$ in \eqref{q0}, we obtain $m\ac{m,0}{m,i}=m\acp{m,1}{m,i+1}$, which yields
\[ \ac{m,0}{m,i}=0\quad\text{for}\quad i\notin\{0;1\}. \]
Finally, taking $i=1$, $j=2$, $k=2$, $n=2m$, and $d=3m$ in \eqref{q0}, we obtain
\[ 2\ac{m,0}{m,1}=\ac{2m,1}{2m,2},\ \text{which implies } \ac{m,0}{m,1}=0. \]

\item[Step 4.] Finally, we show that $\ac{m,i}{m,i}=0$ for all $m,i\in\Z$.
Substituting $d=m+n$ and $k=i+j$ into \eqref{q0}, the equation becomes
\begin{align}
(ni-mj)\ac{m,i}{m,i}+(ni-mj)\ac{n,j}{n,j} &= (ni-mj)\acp{m+n,i+j}{m+n,i+j}, \quad\text{i.e.,}\notag\\
\ac{m,i}{m,i}+\ac{n,j}{n,j} &= \acp{m+n,i+j}{m+n,i+j}\quad\text{for}\quad ni\neq mj. \label{seven}
\end{align}
Substituting successively $(n,j)=(0,1)$, $(n,j)=(1,0)$, and $(n,j)=(-1,0)$ into \eqref{seven}, and taking into account the conditions \eqref{n0}, we obtain the following relations:
\begin{align}
\ac{m,i}{m,i} &= \acp{m,i+1}{m,i+1}\quad\text{for}\quad m\neq0, \label{eight}\\
\ac{m,i}{m,i} &= \acp{m+1,i}{m+1,i}\quad\text{for}\quad i\neq0, \label{nine}\\
\ac{m,i}{m,i} &= \acp{m-1,i}{m-1,i}\quad\text{for}\quad i\neq0. \label{ten}
\end{align}
By analyzing the index shifts in \eqref{eight} and \eqref{nine}, we can establish $\ac{m,i}{m,i}=\ac{1,i+m-1}{1,i+m-1}$.
Furthermore, by combining \eqref{nine} and \eqref{ten}, we get $\ac{m,i}{m,i}=\ac{m+2,i}{m+2,i}$.
From these two equalities, it is not difficult to obtain $\ac{1,k}{1,k}=\ac{1,k+2}{1,k+2}$ for all $k\in\Z$.
Taking into account $\ac{1,0}{1,0}=0$, we get that $\ac{1,k}{1,k}=0$ for all even $k\in\Z$.
Now, setting $n=m=i=1$ and $j=3$ in \eqref{seven}, we get $\ac{1,1}{1,1}+\ac{1,3}{1,3}=\acp{2,4}{2,4}$.
Since $\acp{2,4}{2,4}=\ac{1,4}{1,4}=0$, it yields $\ac{1,3}{1,3}=-\ac{1,1}{1,1}$, and taking into account $\ac{1,3}{1,3}=\ac{1,1}{1,1}$, we deduce $\ac{1,1}{1,1}=0$, which implies $\ac{1,k}{1,k}=0$ for all odd $k\in\Z$; i.e., $\ac{m,i}{m,i}=0$ for all $m,i\in\Z$.
\end{description}
\end{proof}

\subsection{Quasi-derivations of $\B(q)$ for $q\notin\Z$}
\begin{theorem}\label{th8}
If $q\notin\Z$, then $\QDer\big(\B(q)\big)=\langle\ad_{L_{d,k}}\mid d\neq0,\ k\in\Z\rangle\oplus\langle\id\rangle\oplus\langle p_k\rangle\oplus\langle h_{d,k}\rangle$.
\end{theorem}
\begin{proof}
Let $f\in\QDer\big(\B(q)\big)$. Since linear combinations of quasi-derivations are quasi-derivations, we may replace $f$ with the normalized map
\[ \tilde f=f-\sum_{d\neq0}\beta_{d,k}\ad_{L_{d,k}}-\gamma\,\id-\sum_k\lambda_kp_k-\sum_d\mu_{d,k}h_{d,k}. \]
By choosing the coefficients as follows:
\begin{align*}
\beta_{d\neq0,k} &= -\ac{d,k}{0,0}d^{-1}q^{-1};\\
\gamma &= \ac{0,0}{0,0}-q\big(\ac{0,1}{0,1}-\ac{0,0}{0,0}\big);\\
\lambda_k &= \delta_{k,1}\big(\ac{0,1}{0,1}-\ac{0,0}{0,0}\big)+\dhat_{k,1}\ac{0,k-1}{0,0}q^{-1};\\
\mu_{d,k} &= \delta_{d,1}\big(\ac{1,k}{1,0}-\ac{0,k}{0,0}\big)+\dhat_{d,1}\Big(\ac{d,k}{1,0}+\frac{k-(d-2)q}{(d-1)q}\ac{d-1,k}{0,0}\Big),
\end{align*}
we may assume without loss of generality that
\begin{equation}\label{n8}
\ac{0,1}{0,1}=\ac{d,k}{0,0}=\ac{d,k}{1,0}=0\quad\text{for all}\quad d,k\in\Z.
\end{equation}
We now show that \eqref{n8} implies $\ac{d,k}{m,i}=0$ for all $m,i,d,k\in\Z$.

\begin{description}
\item[Step 1.] First, we show that $\ac{d,k}{m,i}=0$ for $d\neq m$.

Substituting $n=j=0$ into \eqref{main}, using $\ac{d,k}{0,0}=0$ and $q\notin\Z$, we obtain
\[ d\ac{d,k}{m,i}=m\acp{d,k}{m,i}, \text{ which in the case of } m=0 \text{ implies } \ac{d,k}{0,i}=0 \text{ for } d\neq0. \]
Letting $n=i=0$ in \eqref{main}, and using $\ac{d-m,k}{0,j}=0$ for $d\neq m$, we have
\[ d\ac{d,k-j}{m,0}=m\acp{d,k}{m,j} \text{ for } d\neq m. \]
Combining these two equalities, we obtain $\ac{d,k}{m,j}=\ac{d,k-j}{m,0}$ for $d\notin\{0;m\}$.
Hence, substituting $m=1$ in this relation, we have $\ac{d,k}{1,j}=\ac{d,k-j}{1,0}=0$ for $d\notin\{0;1\}$.
Now, letting $n=1$, $j=0$ in \eqref{main}, we get
\[ (k+2q-dq)\ac{d-1,k}{m,i}=(i+q-mq)\acp{d,k}{m+1,i},\ \text{which implies} \]
\[ (m+1)(k+2q-dq)\ac{d-1,k-i}{m,0}=d(i+q-mq)\ac{d,k-i}{m+1,0}\quad\text{for}\quad m\neq-1,\ d\notin\{0;1;m+1\}. \]
For $i=0$, $k=t$ and for $i=1$, $k=t+1$, these relations give us the following system of equations:
\[ \begin{cases}
(m+1)(t+2q-dq)\ac{d-1,t}{m,0} = d(q-mq)\ac{d,t}{m+1,0},\\
(m+1)(t+1+2q-dq)\ac{d-1,t}{m,0} = d(1+q-mq)\ac{d,t}{m+1,0}.
\end{cases} \]
Solving this system of equations, we have $\big(t+q(m+1-d)\big)\ac{d-1,t}{m,0}=0$, which implies
\[ \ac{d-1,k}{m,0}=0 \text{ for } m\neq-1,\ d\notin\{1;m+1\};\ \text{i.e., we have } \ac{d,k}{m,i}=0 \text{ for } m\neq-1,\ d\notin\{0;m\}. \]
By choosing $n=d\notin\{0;-1\}$ and $m\notin\{0;n\}$ in \eqref{main}, we derive $\ac{0,k}{m,i}=0$ for $m\notin\{0;-1\}$.
Next, substituting $n=-1$ and $m=2$ into \eqref{main}, we have
\[ \big((d-2)(i+q)-2(k-i+q)\big)\ac{d-2,k-i}{-1,j}=0\quad\text{for}\quad d\neq1. \]
By choosing $i$ such that $\big((d-2)(i+q)-2(k-i+q)\big)\neq0$, we have $\ac{d,k}{-1,j}=0$ for $d\neq-1$.
Hence we deduce $\ac{d,k}{m,i}=0$ for $d\neq m$.

\item[Step 2.] Now, we show that $\ac{m,k}{m,i}=0$ for $k\neq i$.
Substituting $n=j=0$ and $d=m$ in identity \eqref{main} and taking into account \eqref{n8}, we can derive
\[ \ac{m,k}{m,i}=\acp{m,k}{m,i}\quad\text{for}\quad m\neq0. \]
Setting $n=1$, $j=0$, and $d=m+1$ in \eqref{main}, we get
\[ \big((m-1)q-k\big)\ac{m,k}{m,i}=\big((m-1)q-i\big)\acp{m+1,k}{m+1,i}. \]
From these two equalities, we get that
\begin{equation}\label{twelve}
\big((m-1)q-k\big)\ac{m,k}{m,i}=\big((m-1)q-i\big)\ac{m+1,k}{m+1,i}\quad\text{for}\quad m\neq-1.
\end{equation}
By using \eqref{twelve} with $i=0$, for particular values of $m$, we get
\[ \ac{3,k}{3,0}=\frac{q-k}{q}\ac{2,k}{2,0}\quad\text{and}\quad\ac{5,k}{5,0}=\frac{(3q-k)(2q-k)(q-k)}{6q^3}\ac{2,k}{2,0}. \]
Then, substituting the indices $n=3$, $j=0$, $m=2$, $i=0$, and $d=5$ into \eqref{main}, we derive that
\[ (3k+q)\ac{2,k}{2,0}+(q-2k)\ac{3,k}{3,0}=q\ac{5,k}{5,0},\ \text{which implies } \frac{(k+q)(k+2q)(k+3q)}{6q^2}\ac{2,k}{2,0}=0. \]
Since $q\notin\Z$, the factor $k+q\neq0$, and we have $\ac{2,k}{2,0}=0$ for $k\notin\{-2q;-3q\}$.
Consider the following three subcases:
\begin{description}
\item[$(2q,3q\notin\Z)$] If $2q\notin\Z$ and $3q\notin\Z$, then $\ac{2,k}{2,0}=0$ for all $k\in\Z$.
\item[$(2q\in\Z)$] If $2q\in\Z$, then setting $i=0$ and $k=-2q$ in \eqref{twelve}, we have
\[ (m+1)\ac{m,-2q}{m,0}=(m-1)\ac{m+1,-2q}{m+1,0} \text{ for } m\neq-1, \text{ it follows that} \]
\[ \ac{m,-2q}{m,0}=\tfrac{m(m-1)}{2}\ac{2,-2q}{2,0} \text{ for } m\geq0;\qquad \ac{m,-2q}{m,0}=\tfrac{m(m-1)}{2}\ac{-1,-2q}{-1,0} \text{ for } m\leq-1. \]
Putting $n=-1$, $j=0$, $m=2$, $i=0$, and $d=1$, $k=-2q$ in \eqref{main}, we get
\[ \ac{-1,-2q}{-1,0}=\ac{2,-2q}{2,0};\ \text{i.e., } \ac{m,-2q}{m,0}=\tfrac{m(m-1)}{2}\ac{2,-2q}{2,0} \text{ for all } m\in\Z. \]
Next, setting $n=1$, $j=0$, $i=2q$, $k=0$, and $d=m+1$ in \eqref{main}, we get
\[ (1-m)\ac{m,0}{m,2q}=(3-m)\ac{m+1,0}{m+1,2q} \text{ for } m\neq-1, \text{ which implies} \]
\begin{equation}\label{thirteen}
\ac{3,0}{3,2q}=0,\quad \ac{5,0}{5,2q}=3\ac{4,0}{4,2q}\quad\text{and}\quad\ac{6,0}{6,2q}=6\ac{4,0}{4,2q}.
\end{equation}
Then, putting $n=2$, $i=2q$, $j=0$, $k=0$, and $d=m+2$ in \eqref{main}, we get
\[ (2-m)\ac{m,0}{m,2q}+(6+m)\ac{2,-2q}{2,0}=(6-m)\ac{m+2,0}{m+2,2q} \text{ for } m\neq-2, \text{ which yields} \]
\begin{equation}\label{fourteen}
3\ac{2,-2q}{2,0}=\ac{5,0}{5,2q}\quad\text{and}\quad 5\ac{2,-2q}{2,0}-\ac{4,0}{4,2q}=\ac{6,0}{6,2q}.
\end{equation}
Combining the relations \eqref{thirteen} and \eqref{fourteen}, we have $\ac{2,-2q}{2,0}=0$.
\item[$(3q\in\Z)$] If $3q\in\Z$, then setting $i=0$ and $k=-3q$ in \eqref{twelve}, we have
\[ (m+2)\ac{m,-3q}{m,0}=(m-1)\ac{m+1,-3q}{m+1,0} \text{ for } m\neq-1. \]
Similarly to the previous subcase, putting $n=-2$, $j=0$, $m=3$, $i=0$, $d=1$, and $k=-3q$ in \eqref{main}, we get $\ac{-2,-3q}{-2,0}=-\ac{2,-3q}{2,0}$. Consequently,
\[ \ac{m,-3q}{m,0}=\tfrac{(m-1)m(m+1)}{6}\ac{2,-3q}{2,0} \text{ for all } m\in\Z. \]
Next, setting $n=1$, $j=0$, $i=3q$, $k=0$, and $d=m+1$ in \eqref{main}, we obtain
\[ (1-m)\ac{m,0}{m,3q}=(4-m)\ac{m+1,0}{m+1,3q} \text{ for } m\neq-1. \]
Then, setting $n=2$, $i=3q$, $j=0$, $k=0$, and $d=m+2$ in \eqref{main}, we get
\[ (2-m)\ac{m,0}{m,3q}+2(4+m)\ac{2,-3q}{2,0}=(8-m)\ac{m+2,0}{m+2,3q} \text{ for } m\neq-2. \]
Using these two identities for appropriate values of $m$, we obtain
\[ \ac{4,0}{4,3q}=0;\ \ac{6,0}{6,3q}=4\ac{5,0}{5,3q};\ \ac{7,0}{7,3q}=10\ac{5,0}{5,3q};\ 4\ac{2,-3q}{2,0}=\ac{6,0}{6,3q};\ 6\ac{2,-3q}{2,0}-\ac{5,0}{5,3q}=\ac{7,0}{7,3q}. \]
Consequently, it follows that $\ac{2,-3q}{2,0}=0$.
\end{description}
Thus, we have established that $\ac{2,k}{2,0}=0$ for all $k\in\Z$.
After that, putting $n=-1$, $j=0$, $i=0$, $m=2$, and $d=1$ (or $m=3$ and $d=2$) in \eqref{main}, we obtain
\[ (2k+3q)\ac{-1,k}{-1,0}=0 \text{ and } (3k+4q)\ac{-1,k}{-1,0}=0, \text{ which implies } \ac{-1,k}{-1,0}=0 \text{ for all } k\in\Z. \]
Substituting $n=-1$, $j=0$, and $d=m-1$ into the main identity \eqref{main} yields
\begin{equation}\label{fifteen}
\big((m+1)q+k\big)\ac{m,k}{m,i}=\big((m+1)q+i\big)\ac{m-1,k}{m-1,i}\quad\text{for}\quad m\neq1.
\end{equation}
Considering \eqref{twelve} and \eqref{fifteen} as a linear system of equations, we obtain
\[ \ac{m,k}{m,i}=0 \text{ for } m\notin\{-1;0\},\ k\neq-i-3q. \]
Next, setting $k=j-i-3q$, $d=m+n\neq0$, and $m,n\notin\{-1;0\}$, $j\notin\{0;i\}$ in \eqref{main}, we have $\ac{m,k}{m,i}=0$ for $k=-i-3q$.
Finally, considering $m=0$, $n\neq0$ and $m=-1$, $n\notin\{0;1\}$ in \eqref{main}, we conclude that $\ac{0,k}{0,i}=0$ and $\ac{-1,k}{-1,i}=0$. Thus, we have $\ac{m,k}{m,i}=0$ for all $m\in\Z$, $k\neq i$.

\item[Step 3.] Finally, we show that $\ac{m,i}{m,i}=0$ for all $m,i\in\Z$.
Substituting $d=m+n$ and $k=i+j$ into \eqref{main}, the equation becomes
\[ \big(n(i+q)-m(j+q)\big)\ac{m,i}{m,i}+\big(n(i+q)-m(j+q)\big)\ac{n,j}{n,j}=\big(n(i+q)-m(j+q)\big)\acp{m+n,i+j}{m+n,i+j}. \]
Hence, whenever $n(i+q)-m(j+q)\neq0$, we have
\begin{equation}\label{sixteen}
\ac{m,i}{m,i}+\ac{n,j}{n,j}=\acp{m+n,i+j}{m+n,i+j}\quad\text{for}\quad n(i+q)\neq m(j+q).
\end{equation}
Substituting successively $(n,j)=(0,0)$, $(n,j)=(1,0)$, and $(n,j)=(0,1)$ into \eqref{sixteen} and taking into account the conditions \eqref{n8}, we obtain the following relations:
\begin{align}
\ac{m,i}{m,i} &= \acp{m,i}{m,i}\quad\text{for}\quad m\neq0; \label{seventeen}\\
\ac{m,i}{m,i} &= \acp{m+1,i}{m+1,i}\quad\text{for}\quad (m,i)\neq(1,0); \label{eighteen}\\
\ac{m,i}{m,i} &= \acp{m,i+1}{m,i+1}\quad\text{for}\quad m\neq0. \label{nineteen}
\end{align}
Since $\ac{1,0}{1,0}=0$, using \eqref{seventeen} and \eqref{nineteen}, we obtain $\ac{1,i}{1,i}=0$ for all $i\in\Z$.
Moreover, from \eqref{eighteen} and \eqref{nineteen}, we get $\ac{2,0}{2,0}=\acp{2,1}{2,1}=\ac{1,1}{1,1}=0$.
By analyzing the index shifts in \eqref{seventeen} and \eqref{eighteen}, we can establish
\[ \ac{m,0}{m,0}=\ac{2,0}{2,0}=0;\qquad \ac{m,i}{m,i}=\ac{1,i}{1,i}=0 \text{ for } m\geq0;\qquad \ac{m,i}{m,i}=\ac{-1,i}{-1,i} \text{ for } m\leq-1. \]
Finally, putting $m=-1$, $n=2$, and $j=0$ in \eqref{sixteen}, we obtain
\[ \ac{-1,i}{-1,i}+\ac{2,0}{2,0}=\acp{1,i}{1,i},\ \text{which implies } \ac{-1,i}{-1,i}=0 \text{ for } 2i+3q\neq0. \]
If $2i+3q=0$ (which is possible only if $3q\in\Z$), then $3i+4q=-\frac{q}{2}\neq0$, and putting $m=-1$, $n=3$, and $j=0$ in \eqref{sixteen}, we obtain $\ac{-1,i}{-1,i}+\ac{3,0}{3,0}=\acp{2,i}{2,i}=\ac{2,i}{2,i}=0$, i.e., again $\ac{-1,i}{-1,i}=0$. Hence $\ac{m,i}{m,i}=0$ for all $m,i\in\Z$.
\end{description}
\end{proof}

\subsection{Quasi-derivations of $\B(q)$ for $q\in\Z^{*}$}\label{th9-sub}
\begin{theorem}\label{th9}
If $q\in\Z^{*}$, then
\[ \QDer\big(\B(q)\big)=\langle\ad_{L_{d,k}}\mid d\neq0,\ k\in\Z\rangle\oplus\langle\id\rangle\oplus\langle p_k\rangle\oplus\langle h_{d,k}\rangle\oplus\AnnQ\big(\B(q)\big). \]
\end{theorem}
\begin{proof}
Let $f\in\QDer\big(\B(q)\big)$. Since linear combinations of quasi-derivations are quasi-derivations, we may replace $f$ with the normalized map
\[ \tilde f=f-\sum_{d\neq0}\beta_{d,k}\ad_{L_{d,k}}-\gamma\,\id-\sum_k\lambda_kp_k-\sum_d\mu_{d,k}h_{d,k}-\sum_{d,k}\sigma_{d,k}g_{d,k}, \]
by choosing the coefficients as follows:
\begin{align*}
\beta_{d\neq0,k} &= -\ac{d,k}{0,0}d^{-1}q^{-1};\\
\gamma &= \ac{0,0}{0,0}-q\big(\ac{1,1}{1,1}-\ac{1,0}{1,0}\big);\\
\lambda_k &= \delta_{k,1}\big(\ac{1,1}{1,1}-\ac{1,0}{1,0}\big)+\delta_{k,1-q}\big(\ac{1,1-q}{1,1}-\ac{1,-q}{1,0}\big)+\dhat_{k,1}\dhat_{k,1-q}\ac{0,k-1}{0,0}q^{-1};\\
\mu_{d,k} &= \delta_{d,1}\delta_{k,-q}\big((1+q)\ac{1,-q}{1,0}-q\ac{1,1-q}{1,1}\big)+\delta_{d,0}\delta_{k,-q}\Big(\ac{0,1-q}{1,1}-\frac{(q+1)\ac{-1,-q}{0,0}}{q}\Big)\\
&\quad+\delta_{d,1}\dhat_{k,-q}\big(\ac{1,k}{1,0}-\ac{0,k}{0,0}\big)+\dhat_{d,1}\dhat_{(d,k),(0,-q)}\Big(\ac{d,k}{1,0}+\frac{k-(d-2)q}{(d-1)q}\ac{d-1,k}{0,0}\Big);\\
\sigma_{d,k} &= \delta_{d,0}\delta_{k,0}\big(\ac{0,-q}{0,0}-q\lambda_{1-q}\big)+\delta_{d,0}\delta_{k,-q}\big(\ac{0,-q}{0,-q}-\gamma\big)\\
&\quad+\delta_{d,0}\dhat_{k,0}\dhat_{k,-q}\big(\ac{0,-q}{0,k}-(k+q)\lambda_{1-k-q}\big)+\dhat_{d,0}\big(\ac{0,-q}{d,k}-dq\beta_{-d,-k-q}-d\mu_{1-d,-k-q}\big);
\end{align*}
we may assume without loss of generality that
\begin{equation}\label{n9}
\ac{1,1}{1,1}=\ac{0,1-q}{1,1}=\ac{1,1-q}{1,1}=\ac{d,k}{0,0}=\ac{d,k}{1,0}=\ac{0,-q}{d,k}=0\quad\text{for all}\quad d,k\in\Z.
\end{equation}
We now show that \eqref{n9} implies $\ac{d,k}{m,i}=0$ for all $m,i,d,k\in\Z$.

\begin{description}
\item[Step 1.] First, we show that $\ac{d,k}{m,i}=0$ for $d\neq m$.
Substituting $n=j=0$ into \eqref{main} yields
\begin{equation}\label{twentyone}
d\ac{d,k}{m,i}=m\acp{d,k}{m,i}.
\end{equation}
Substituting $m=0$ into this relation implies that $\ac{d,k}{0,i}=0$ for $d\neq0$.
Letting $n=i=0$ in \eqref{main}, and using $\ac{d,k}{0,i}=0$, we get $d(j+q)\ac{d,k-j}{m,0}=m(j+q)\acp{d,k}{m,j}$, which implies
\begin{equation}\label{twentytwo}
d\ac{d,k-j}{m,0}=m\acp{d,k}{m,j}\quad\text{for}\quad d\neq m,\quad j\neq-q.
\end{equation}
Combining \eqref{twentyone} and \eqref{twentytwo}, we obtain
\begin{equation}\label{twentythree}
\ac{d,k}{m,i}=\ac{d,k-i}{m,0}\quad\text{for}\quad d\notin\{0;m\},\quad i\neq-q.
\end{equation}
Substituting $m=1$ into \eqref{twentythree}, we obtain $\ac{d,k}{1,i}=\ac{d,k-i}{1,0}=0$ for $d\notin\{0;1\}$, $i\neq-q$.
Next, by substituting $n=1$ and $j=0$ into \eqref{main}, using $\ac{d,k}{1,0}=0$ and relations \eqref{twentytwo} and \eqref{twentythree}, for $m\neq-1$, $d\notin\{0;1;m+1\}$, $i\neq-q$, we have
\[ (m+1)(k+2q-dq)\ac{d-1,k-i}{m,0}=d(i+q-mq)\ac{d,k-i}{m+1,0}. \]
For $i=0$, $k=t$ and for $i=1$, $k=t+1$, these relations give us the following system of equations:
\[ \begin{cases}
(m+1)(t+2q-dq)\ac{d-1,t}{m,0} = d(q-mq)\ac{d,t}{m+1,0},\\
(m+1)(t+1+2q-dq)\ac{d-1,t}{m,0} = d(1+q-mq)\ac{d,t}{m+1,0}.
\end{cases} \]
Solving this system of equations, we have that
\[ \ac{d-1,t}{m,0}=\ac{d,t}{m+1,0}=0 \text{ for } t+q(m+1-d)\neq0, \]
which implies $\ac{d,k}{m,0}=0$ for $m\neq-1$, $d\notin\{0;m\}$. Thus, we have that
\[ \ac{d,k}{m,i}=\ac{d,k-i}{m,0}=0 \text{ for } m\neq-1,\ d\notin\{0;m\},\ i\neq-q. \]
Choosing $d=n\notin\{0;-1\}$ and $m\notin\{0;n\}$ in \eqref{main}, it follows that
\[ \ac{0,k}{m,i}=0 \text{ for } m\notin\{0;-1\},\ i\neq-q. \]
Substituting $n=-1$, $m=2$ into \eqref{main}, then applying \eqref{twentytwo}, we obtain
\[ \big((d-2)(i+q)-2(k-i+q)\big)\ac{d-2,k-i}{-1,j}=0, \]
and by choosing $i$ such that the coefficient $(d-2)(i+q)\neq2(k-i+q)$, we obtain
\[ \ac{d-2,k-i}{-1,j}=0 \text{ for } d\neq1. \]
Consequently, $\ac{d,k}{-1,j}=\ac{d,k-j}{-1,0}=0$ for $d\neq-1$, $j\neq-q$.
Next, setting $i=-q$ and $j=2q$ in \eqref{main}, we derive
\[ \big(n(k-q)-3q(d-n)\big)\ac{d-n,k-2q}{m,-q}=0 \text{ for } d\neq m+n. \]
By choosing $n$ such that $n(k-q)\neq3q(d-n)$, we obtain $\ac{d-n,k-2q}{m,-q}=0$ for $d\neq m$.
Hence, we deduce $\ac{d,k}{m,i}=0$ for all $m\in\Z$ and $d\neq m$.

\item[Step 2.] Second, we show that $\ac{m,k}{m,i}=0$ for $k\neq i$.
The proof of this case is similar to that of Step 2 in Theorem \ref{th8}; we highlight the key differences below. First, considering $n=j=0$ and $n=1$, $j=0$ in \eqref{main}, we get
\[ \big((m-1)q-k\big)\ac{m,k}{m,i}=\big((m-1)q-i\big)\ac{m+1,k}{m+1,i} \text{ for } m\neq-1. \]
Considering $i=0$, for particular values of $m$, we express $\ac{3,k}{3,0}$ and $\ac{5,k}{5,0}$ in terms of $\ac{2,k}{2,0}$. Then, substituting $n=3$, $j=0$, $m=2$, $i=0$, and $d=5$ into \eqref{main}, we obtain
\[ \frac{(k+q)(k+2q)(k+3q)}{6q^2}\ac{2,k}{2,0}=0. \]
Similarly to the proof of Theorem \ref{th8}, this yields the relation
\[ \ac{2,k}{2,0}=0 \text{ for } k\notin\{-q;-2q;-3q\}. \]
The cases $k=-2q$ and $k=-3q$ are treated literally as the subcases $(2q\in\Z)$ and $(3q\in\Z)$ in the proof of Theorem \ref{th8}.
Since $q\in\Z$, unlike in the proof of Theorem \ref{th8}, it is necessary to explicitly show that $\ac{2,-q}{2,0}=0$. For this purpose, setting $i=0$ and $k=-q$ in \eqref{twelve}, we have
\[ m\ac{m,-q}{m,0}=(m-1)\ac{m+1,-q}{m+1,0} \text{ for } m\neq-1. \]
It follows that
\[ \ac{m,-q}{m,0}=(m-1)\ac{2,-q}{2,0} \text{ for } m\geq1;\qquad \ac{m,-q}{m,0}=\tfrac{1-m}{2}\ac{-1,-q}{-1,0} \text{ for } m\leq-1. \]
Putting $n=-1$, $j=0$, $m=2$, $i=0$, $d=1$, and $k=-q$ in \eqref{main}, we get $\ac{-1,-q}{-1,0}=-2\ac{2,-q}{2,0}$. Hence, we obtain $\ac{m,-q}{m,0}=(m-1)\ac{2,-q}{2,0}$ for all $m\in\Z$.
Then, putting $m=4$, $n=2$, $i=q$, $j=0$, $k=0$, and $d=6$ in \eqref{main}, we get $\ac{4,0}{4,q}=0$.
Next, setting $m=2$, $n=2$, $i=q$, $j=0$, $k=0$, and $d=4$ in \eqref{main}, we get $\ac{2,-q}{2,0}=0$.
Thus, we have established that $\ac{2,k}{2,0}=0$ for all $k\in\Z$.
Proceeding analogously to the proof of Theorem \ref{th8}, by evaluating the main identity \eqref{main} for suitable choices of indices such as $n=-1$, $j=0$, $i=0$, and $m=2$, $d=1$ (or $m=3$, $d=2$) and combining the resulting relations, we obtain
\[ (2k+3q)\ac{-1,k}{-1,0}=0 \text{ and } (3k+4q)\ac{-1,k}{-1,0}=0, \text{ which implies } \ac{-1,k}{-1,0}=0 \text{ for all } k\in\Z. \]
Substituting $n=1$, $j=0$, $d=m+1$ and $n=-1$, $j=0$, $d=m-1$ into the main identity \eqref{main} yields
\begin{align*}
\big((m-1)q-k\big)\ac{m,k}{m,i} &= \big((m-1)q-i\big)\ac{m+1,k}{m+1,i}\quad\text{for}\quad m\neq-1,\\
\big((m+1)q+k\big)\ac{m,k}{m,i} &= \big((m+1)q+i\big)\ac{m-1,k}{m-1,i}\quad\text{for}\quad m\neq1.
\end{align*}
Considering these equations as a linear system of equations, we obtain
\[ \ac{m,k}{m,i}=0 \text{ for } m\notin\{-1;0\},\ k\neq-i-3q. \]
Similarly, for $k=j-i-3q$, by applying the same argument as in the proof of the above theorem, we obtain $\ac{m,k}{m,i}=0$.
Finally, substituting $m=0,n\neq0$ and $m=-1,n\neq0,1$ into \eqref{main} yields $\ac{0,k}{0,i}=0$ and $\ac{-1,k}{-1,i}=0$. Hence, we obtain $\ac{m,k}{m,i}=0$ for any $m\in\Z$ and $k\neq i$.

\item[Step 3.] As a final step, we show that $\ac{m,i}{m,i}=0$ for all $m,i\in\Z$.
Substituting $d=m+n$ and $k=i+j$ into \eqref{main}, the equation becomes
\[ \ac{m,i}{m,i}+\ac{n,j}{n,j}=\acp{m+n,i+j}{m+n,i+j} \text{ for } ni-mj+(n-m)q\neq0. \]
Substituting successively $(n,j)=(0,0)$, $(n,j)=(1,0)$, and $(n,j)=(1,1)$ into this equation, we obtain
\begin{align*}
\ac{m,i}{m,i} &= \acp{m,i}{m,i} &&\text{for } m\neq0;\\
\ac{m,i}{m,i} &= \acp{m+1,i}{m+1,i} &&\text{for } i\neq q(m-1);\\
\ac{m,i}{m,i} &= \acp{m+1,i+1}{m+1,i+1} &&\text{for } i\neq q(m-1)+m.
\end{align*}
As established at the beginning of the proof, we have $\ac{0,-q}{0,-q}=0$. Proceeding as in Step 3 of the proof of Theorem \ref{th8}, we obtain $\ac{m,i}{m,i}=0$ for all $i\neq q(m-1)$.
It remains to consider the case $i=q(m-1)$, $m\neq0$. 

(The argument via the third relation above does not work for $q=1$, since then $\big(m+1,q(m-1)+1\big)=\big(m+1,qm\big)$ is again an exceptional pair, and for $m=-1$.) If $m\neq-2$, we take $(n,j)=(2,0)$: then $ni-mj+(n-m)q=qm\neq0$, and since $\ac{2,0}{2,0}=0$ and $\acp{m+2,q(m-1)}{m+2,q(m-1)}=\ac{m+2,q(m-1)}{m+2,q(m-1)}=0$ (because $q(m-1)\neq q(m+1)$), we obtain $\ac{m,q(m-1)}{m,q(m-1)}=0$. If $m=-2$, we take $(n,j)=(3,0)$: then $ni-mj+(n-m)q=-4q\neq0$, and $\ac{3,0}{3,0}=0$, $\acp{1,-3q}{1,-3q}=\ac{1,-3q}{1,-3q}=0$ give $\ac{-2,-3q}{-2,-3q}=0$.
Therefore, we conclude that $\ac{m,i}{m,i}=0$ for all $m,i\in\Z$.
\end{description}
\end{proof}

\subsection{Corollaries}
Summarizing the results obtained in Subsections \ref{th7-sub}--\ref{th9-sub}, we have the following statements.

\begin{maintheorem}
$\QDer\big(\B(q)\big)=\langle\ad_{L_{d,k}}\mid d\neq0,\ k\in\Z\rangle\oplus\langle\id\rangle\oplus\langle p_k\rangle\oplus\langle h_{d,k}\rangle\oplus\AnnQ\big(\B(q)\big)$\footnote{$\AnnQ\big(\B(q)\big)=0$ if $q\notin\Z$.}.
\end{maintheorem}

Here we write the inner derivations in the form $\langle\ad_{L_{d,k}}\mid d\neq0\rangle$ in order to make the sum direct: since $\ad_{L_{0,k}}=(k+q)h_{1,k}$, we have $\Inn\big(\B(q)\big)=\langle\ad_{L_{d,k}}\mid d\neq0\rangle\oplus\langle h_{1,k}\mid k\neq-q\rangle$.
In the Main Theorem, only finite linear combinations of $\ad_{L_{d,k}}$, $\id$, $p_k$, and $h_{d,k}$ occur (each quasi-derivation is a linear map of $\B(q)$ into itself), while an element of $\AnnQ\big(\B(q)\big)$ may be an infinite formal sum.

\begin{definition}
A nonzero $\delta$-derivation $\varphi$ is nontrivial if $\delta\notin\{0;1\}$ and $\varphi\notin\langle\id\rangle$.
\end{definition}

\begin{corollary}\label{cor11}
If $q\notin\Z$, then $\B(q)$ does not admit nontrivial $\delta$-derivations.
\end{corollary}

\begin{corollary}\label{cor12}
If $q\in\Z$ and $\varphi$ is a nontrivial $\delta$-derivation of $\B(q)$, then $\varphi\in\langle D_5\rangle$ if $\delta\neq\frac12$ and $\varphi\in\langle\id,D_5\rangle$ if $\delta=\frac12$, where
\[ D_5(L_{0,-2q})=L_{0,-q}. \]
\end{corollary}

\begin{definition}
A bilinear map $\varphi$ on a Lie algebra $(\mathfrak{L},[\cdot,\cdot])$ is called a $\delta$-biderivation if
\begin{align*}
\varphi\big(x,[y,z]\big) &= \delta\big([\varphi(x,y),z]+[y,\varphi(x,z)]\big),\\
\varphi\big([x,y],z\big) &= \delta\big([\varphi(x,z),y]+[x,\varphi(y,z)]\big).
\end{align*}
A nonzero $\delta$-biderivation $\varphi$ is called nontrivial if $\delta\notin\{0;1\}$ and $\varphi\neq\lambda[\cdot,\cdot]$, where $\lambda\in\C$.
\end{definition}

\begin{corollary}
If $\varphi$ is a nontrivial $\delta$-biderivation on $\B(q)$, then $q\in\Z$ and
\[ \varphi\big(L_{0,-2q},L_{0,-2q}\big)=\mu L_{0,-q},\quad \mu\in\C^{*}, \]
and $\varphi$ vanishes on all other pairs of basis elements.
\end{corollary}

\begin{corollary}\label{cor15}
If $\cdot$ defines a nontrivial transposed $\delta$-Poisson structure on $\B(q)$ with $\delta\notin\{0;1\}$, then $q\in\Z$ and $L_{0,-2q}\cdot L_{0,-2q}=\mu L_{0,-q}$, $\mu\in\C^{*}$, and all other products of basis elements are zero. In particular, up to isomorphism, there is only one nontrivial transposed $\delta$-Poisson structure $\cdot$ on $\B(q)$, $q\in\Z$, $\delta\notin\{0;1\}$; it is given by $L_{0,-2q}\cdot L_{0,-2q}=L_{0,-q}$.
\end{corollary}
\begin{proof}
By Lemma~\ref{lem4}, every right multiplication $R_z$ is a $\delta$-derivation, so the statement follows from Corollaries~\ref{cor11} and~\ref{cor12} and the commutativity of $\cdot$ (associativity is not needed for this part). Conversely, since $L_{0,-q}$ is central and $L_{0,-2q}\notin[\B(q),\B(q)]$, the product $L_{0,-2q}\cdot L_{0,-2q}=\mu L_{0,-q}$ satisfies \eqref{tdp} for every $\delta$. For $q\neq0$, the automorphism $\theta\colon L_{m,i}\mapsto b^{i}L_{m,i}$ of $\B(q)$ with $b^{3q}=\mu^{-1}$ satisfies $\theta(x\cdot_\mu y)=\theta(x)\cdot_1\theta(y)$, where $\cdot_\mu$ denotes the structure with parameter $\mu$; for $q=0$, the same is done by the automorphism $L_{0,0}\mapsto\mu L_{0,0}$, $L_{m,i}\mapsto L_{m,i}$ for $(m,i)\neq(0,0)$ (recall that $L_{0,0}$ is central and $L_{0,0}\notin[\B(0),\B(0)]$).
\end{proof}

\begin{remark}
The restriction $\delta\notin\{0;1\}$ in Corollary~\ref{cor15} cannot be removed. For example, let $\delta=0$ and $2q=n\in\Z$, and put $e=L_{0,-n}$. Then $e\notin[\B(q),\B(q)]$, so the product $e\cdot e=e$ (all other products of basis elements being zero) is commutative and associative and satisfies $x\cdot[y,z]=0$, i.e., it is a nontrivial transposed $0$-Poisson structure on $\B(q)$. For $q=\frac12$ this contradicts the conclusion $q\in\Z$, and for $q\in\Z^{*}$ this structure is not of the form $L_{0,-2q}\cdot L_{0,-2q}=\mu L_{0,-q}$. For $\delta=1$ (i.e., Poisson structures), Lemma~\ref{lem4} together with the description of derivations of $\B(q)$ would be needed instead of Corollaries~\ref{cor11} and~\ref{cor12}.
\end{remark}

\end{document}